\pdfoutput=1
\RequirePackage{ifpdf}
\ifpdf 
\documentclass[pdftex]{sigma}
\else
\documentclass{sigma}
\fi

\numberwithin{equation}{section}

\newtheorem{Theorem}{Theorem}[section]
\newtheorem{Lemma}[Theorem]{Lemma}
 { \theoremstyle{definition}
\newtheorem{Example}[Theorem]{Example}}

\begin{document}

\allowdisplaybreaks

\newcommand{\arXivNumber}{1803.03105}

\renewcommand{\PaperNumber}{112}

\FirstPageHeading

\ShortArticleName{Strictly Positive Definite Functions on Compact Two-Point Homogeneous Spaces}

\ArticleName{Strictly Positive Definite Functions\\ on Compact Two-Point Homogeneous Spaces:\\ the Product Alternative}

\Author{Rafaela N.~BONFIM~$^\dag$, Jean C.~GUELLA~$^\ddag$ and Valdir A.~MENEGATTO~$^\ddag$}

\AuthorNameForHeading{R.N.~Bonfim, J.C.~Guella and V.A.~Menegatto}

\Address{$^\dag$~DEMAT-Universidade Federal de S\~{a}o Jo\~{a}o Del Rei, Pra\c{c}a Frei Orlando, 170, Centro,\\
\hphantom{$^\dag$}~36307-352 S\~{a}o Jo\~{a}o del Rei - MG, Brazil}
\EmailD{\href{rafaelabonfim@ufsj.edu.br}{rafaelabonfim@ufsj.edu.br}}

\Address{$^\ddag$~Instituto de Ci\^encias Matem\'aticas e de Computa\c{c}\~ao, Universidade de S\~ao Paulo, \\
\hphantom{$^\ddag$}~Caixa Postal 668, 13560-970, S\~ao Carlos - SP, Brazil}
\EmailD{\href{mailto:jeanguella@gmail.com}{jeanguella@gmail.com}, \href{menegatt@icmc.usp.br}{menegatt@icmc.usp.br}}

\ArticleDates{Received March 08, 2018, in final form October 10, 2018; Published online October 16, 2018}

\Abstract{For two continuous and isotropic positive definite kernels on the same compact two-point homogeneous space, we determine necessary and sufficient conditions in order that their product be strictly positive definite. We also provide a similar characterization for kernels on the space-time setting $G \times S^d$, where $G$ is a locally compact group and $S^d$ is the unit sphere in $\mathbb{R}^{d+1}$, keeping isotropy of the kernels with respect to the $S^d$ component. Among other things, these results provide new procedures for the construction of valid models for interpolation and approximation on compact two-point homogeneous spaces.}


\Classification{33C45; 42A82; 42C10; 43A35}

\section{Introduction}\label{s-introd}

Positive definite functions and kernels on manifolds have special importance for probability theory, approximation theory, spatial statistics and stochastic processes. In applications, the case in which the manifold is a 2-dimensional sphere is the most common one, once the sphere plays the surface of the Earth in many mathematical models. Originally, positive definite functions were studied within the scope of harmonic analysis, distance geometry and the theory of integral equations by Bochner, Schoenberg and Young, among others. The most relevant contributions which are related to this paper are Schoenberg's characterization of positive definite functions on spheres given in~\cite{schoen} and Gangolli's extension to all the compact two-point homogeneous spaces described in~\cite{gangolli}. Some other references will be quoted at the opportune time.

Let $\mathbb{M}^d$ denote a $d$-dimensional compact two-point homogeneous space. As pointed by Wang~\cite{wang}, $\mathbb{M}^d$ belongs to one of the following classes: the unit circle $S^1$, higher dimensional unit spheres $S^d$, $d=2,3\ldots$, the real projective spaces $\mathbb{P}^d(\mathbb{R})$, $d=2,3,\ldots$, the complex projective spaces $\mathbb{P}^d(\mathbb{C})$, $d=4,6,\ldots$, the quaternionic projective spaces $\mathbb{P}^d(\mathbb{H})$, $d=8,12,\ldots$, and the Cayley projective plane $\mathbb{P}^{d}({\rm Cay})$, $d=16$. These manifolds are metric spaces when endowed with their usual Riemannian (geodesic) distance. Additional properties of compact two-point homogeneous spaces can be found in the textbooks \cite{helgason,wolf}. If $x$ and $y$ are two elements of~$\mathbb{M}^d$, we will write~$|xy|$ to indicate the distance between them. In order to make the treatment uniform, the distance will be normalized so that $|xy|$ is at most $2\pi$, no matter what $x$, $y$ and~$\mathbb{M}^d$ are. That being said, the positive definite kernels on $\mathbb{M}^d$ to be considered here are of the form
\begin{gather}\label{PDK}
K(x,y)=f(\cos{(|xy|/2)}), \qquad x,y \in \mathbb{M}^d,
\end{gather}
in which $f$ is a continuous function with domain $[-1,1]$. It is not uncommon to call $f$ the {\em isotropic part} of the kernel $K$. Indeed, the manifolds $\mathbb{M}^d$ possess a group of motions $G_d$ which takes $(x,y)\in \mathbb{M}^d \times \mathbb{M}^d$ to $(z,w)\in \mathbb{M}^d \times \mathbb{M}^d$ when $|xy|=|zw|$. In particular, a kernel as above is {\em isotropic} in the sense that
\begin{gather*}
K(x,y)=K(Ax,Ay), \qquad x,y \in \mathbb{M}^d, \qquad A \in G_d.
\end{gather*}
The positive definiteness of $K$ demands that for any positive integer $n$ and any distinct points $x_1, x_2, \ldots, x_n$ on $\mathbb{M}^d$, the $n \times n$ matrix with entries $K(x_i, x_j)$ is nonnegative definite. That corresponds to
\begin{gather}\label{ineq}
\sum_{i,j=1}^n c_i c_j K(x_i,x_j) \geq 0,
\end{gather}
for any real numbers $c_1, c_2, \ldots, c_n$.

According to \cite{gangolli,schoen}, a kernel $K$ as in (\ref{PDK}) is positive definite if, and only if, its isotropic part $f$ has a Fourier--Jacobi series representation in the form
\begin{gather}\label{PDM}
f(t)=\sum_{k=0}^{\infty}a_k(f) P_k^{(\alpha,\beta)}(t), \qquad t \in [-1,1],
\end{gather}
in which all the coefficients $a_k(f)$ are nonnegative, $P_k^{(\alpha,\beta)}$ is the usual Jacobi polynomial of degree $k$ associated with the pair $(\alpha,\beta)$~\cite{szego}, and the series is convergent at $t=1$. The first upper exponent~$\alpha$ depends only on the dimension~$d$ of~$\mathbb{M}^d$ in the sense that $\alpha:=(d-2)/2$, whereas~$\beta$ can take the values $(d-2)/2$, $-1/2$, $0$, $1$, $3$, depending on the respective category~$\mathbb{M}^d$ belongs to, among those stressed by Wang. The coefficients~$a_k(f)$ depend upon~$\alpha$ and~$\beta$ but that will not be emphasized in our notation. Obviously, the series representation mentioned above does not depend upon the particular normalization adopted for the Jacobi polynomials.

The strict positive definiteness of a positive definite kernel as above deserves attention when interpolation procedures need to be solved. It demands strict inequalities in~(\ref{ineq}) when the scalars~$c_i$ are nonzero. In statistical language, the strict positive definiteness of a covariance function (positive definite kernel) provides invertible kriging coefficient matrices and, therefore, the existence of a unique solution for the associated kriging system. The characterization of strictly positive definite kernels and the construction of strictly positive definite kernels featuring special needs is of practical relevance not only in statistics but also in approximation theory.

The characterization for strict positive definiteness within Gangolli's class was achieved in recent years. It begins with the observation that strict positive definiteness depends upon the set $\{k\colon a_k(f)>0\}$ attached to the isotropic part of the positive definite kernel and not on the actual values of $a_k(f)$ themselves. Precisely, the following result holds (see~\cite{barbosa,chen,menega}).

\begin{Theorem} \label{SPD}Let $f$ possess a Fourier--Jacobi series representation as in \eqref{PDM}. It is the isotropic part of a strictly positive definite kernel on~$\mathbb{M}^d$ if, and only if, the respective conditions given below hold:
\begin{enumerate}\itemsep=0pt
\item[$(i)$] $\mathbb{M}^d=S^1$: the set $\{k\colon a_{|k|}(f)>0\}$ intersects every full arithmetic progression in $\mathbb{Z}$, that is, all the sets $n\mathbb{Z} +j=\{nl+j\colon l\in \mathbb{Z}\}$, where $n,j \in \mathbb{N}$ and $n\geq 2$.
\item[$(ii)$] $\mathbb{M}^d=S^d$, $d\geq 2$: the set $\{k\colon a_{k}(f)>0\}$ contains infinitely many even and infinitely many odd integers.
\item[$(iii)$] $\mathbb{M}^d\neq S^d$, $d\geq 1$: the set $\{k\colon a_{k}(f)>0\}$ contains infinitely many integers, that is, $f$ is not a polynomial.
\end{enumerate}
\end{Theorem}

The classes of positive definite kernels introduced so far are closed under linear combinations with nonnegative coefficients and finite products. The first assertion follows from the definition of positive definiteness while the other one follows from the Schur product theorem \cite[p.~455]{horn}. Regarding strict positive definiteness with respect to linear combinations, if~$f$ and~$g$ are functions possessing a Fourier--Jacobi series representation as in (\ref{PDM}), for the same~$\mathbb{M}^d$ and~$a$ and~$b$ are nonnegative real numbers, then the function $af+bg$ is the isotropic part of a~strictly positive definite kernel on~$\mathbb{M}^d$ if, and only if, the sets $\{k\colon a a_{|k|}(f)+b a_{|k|}(g)>0\}$ and \smash{$\{k\colon a a_{k}(f)+ba_k(g)>0\}$} satisfy the corresponding conditions in Theorem~\ref{SPD}, that is, if, and only if, the sets $\{k\colon aa_{|k|}(f)>0 \} \cup \{ l\colon ba_{|l|}(g)> 0\}$ and $\{k\colon aa_{k}(f)>0 \} \cup \{ l\colon ba_{l}(g)> 0\}$ satisfy the respective conditions in Theorem~\ref{SPD}.

Having said that, the focus in the first half of this paper will be the analysis of strict positive definiteness of product covariance models on these same spaces. For a fixed space~$\mathbb{M}^d$ and two continuous functions $f,g\colon [-1,1] \to \mathbb{R}$ which are isotropic parts of two positive definite kernels on~$\mathbb{M}^d$, we will find necessary and sufficient conditions on them in order that the product~$fg$ be the isotropic part of a strictly positive definite kernel on~$\mathbb{M}^d$. The problem can be seen as a~particular formulation of what was called $DC$-strict positive definiteness in a product space in \cite{barbosa0,guella1}.

Oppenheim's inequality \cite[p.~480]{horn} is all that is needed in order to see that if one of the functions is nonzero and the other is the isotropic part of a strictly positive definite kernel on~$\mathbb{M}^d$, then~$fg$ is the isotropic part of a strictly positive definite kernel on $\mathbb{M}^d$. However, the product of two non strictly positive definite kernels on~$\mathbb{M}^d$ may be strictly positive definite, as the example~$\mathbb{M}^d=S^d$, $d\geq 2$, and the coefficients
\begin{gather*}a_{0}(f)=a_1(f)=1,\qquad a_k(f)=0, \qquad k\neq 0,1,\end{gather*}
and
\begin{gather*}a_{2k+1}(g)=3^{-2k-1},\qquad a_{2k}(g)=0,\qquad k=0,1,\ldots,\end{gather*}
show. Indeed, it is a direct consequence of the theorems to be proved in Section~\ref{s1} that $\{k\colon a_k(fg)>0\}$ contains infinitely many even and infinitely many odd integers. In particular, the question to be analized in this paper is, indeed, nontrivial. As a bypass, in the second half of the paper, we will also consider the very same problem adapted to other manifolds: a group cross a high-dimensional sphere, keeping isotropy of the kernel with respect to the spherical component and the unit sphere in~$\mathbb{C}^q$.

The other sections in the paper are organized as follows. In Section~\ref{s1}, we will solve the problem proposed above in the case of a compact two-point homogeneous space~$\mathbb{M}^d$. In Section~\ref{extend}, we extend the results from Section~\ref{s1} for intersection classes of positive definite kernels taking into account original characterizations for the classes in~\cite{mene,schoen}. In Section~\ref{produto}, we extrapolate the problem to space-time kernels, that is, complex kernels on $G\times S^d$, in which~$G$ is a locally compact group, adopting the context for positive definiteness presented in \cite{berg,guella}. In particular, isotropy for the~$S^d$ component of the kernel will be kept in the analysis. Finally, in Section~\ref{complex}, we consider the very same problem now adapted to complex kernels on the unit sphere in~$\mathbb{C}^q$.

\section{The main results for homogeneous spaces} \label{s1}

We begin recalling a general linearization formula for Jacobi polynomials \cite[p.~41]{askey}. For a~fi\-xed~$\mathbb{M}^d$ and two functions with Fourier--Jacobi series representation as in (\ref{PDM}), it implies a~linearization in the series expansion of the product of the functions.

\begin{Lemma} \label{gasp} If $\alpha\geq \beta>-1$ and
\begin{gather*}(\alpha+\beta+1)(\alpha+\beta+4)^2(\alpha+\beta+6)\geq (\alpha-\beta)^2\big[(\alpha+\beta+1)^2-7(\alpha+\beta+1)-24\big],\end{gather*}
then
\begin{gather*}
P_k^{(\alpha,\beta)}(t)P_l^{(\alpha,\beta)}(t)=\sum_{\mu=|k-l|}^{k+l} b_{k,l}^{\alpha,\beta}(\mu) P_{\mu}^{(\alpha,\beta)}(t), \qquad t \in [-1,1],\end{gather*}
in which all the coefficients $b_{k,l}^{\alpha,\beta}(\mu)$ are nonnegative. The coefficient $ b_{k,l}^{\alpha,\beta}(k+l)$ is, in fact, positive.
\end{Lemma}
\begin{proof}The recurrence relation for the product was set up by Hylleraas in \cite{hylle} while the nonnegativity of the coefficients was obtained by Gasper in \cite{gasper,gasper1}. The last statement in the lemma is implicit in the proof of Theorem~1 in~\cite{gasper}.
\end{proof}

It is an easy matter to see that the nonnegativity of the coefficients is granted in the cases in which $\alpha\geq \beta>-1$ and $\alpha+\beta\geq -1$. In particular, it is also granted in the cases in which $\alpha=(d-2)/2$ and $\beta=-1/2, 0, 1, 3$ covered by Wang's classification. Therefore, for all the compact two-point homogeneous spaces, the class of functions possessing a representation as in~(\ref{PDM}) is a semigroup under pointwise multiplication.

The Jacobi polynomial $P_k^{(-1/2,-1/2)}$ is a positive multiple of the Chebyshev polynomial~$T_k$ of the first kind. Normalizing the Chebyshev polynomials by $T_k(1)=1$, $k=0,1,\ldots$, we have that
\begin{gather*}T_k(t)=\cos (k \arccos t),\qquad t\in [-1,1], \qquad k=0,1,\ldots,\end{gather*}
and the linearization formula described in Lemma~\ref{gasp} reduces itself to the cosine addition formula
\begin{gather*}T_k(t)T_l(t)=\frac{1}{2}[T_{k+l}(t)+T_{|k-l|}(t)],\qquad t \in [-1,1], \qquad k,l\in \mathbb{Z}_+.\end{gather*}
Now, if $f$ and $g$ are the isotropic parts of positive definite kernels on $S^1$ with Fourier--Jacobi series as in (\ref{PDM}), we can replace $P_k^{(-1/2,-1/2)}$ with $T_k$ and write
\begin{gather*}
f(t)g(t) = \sum_{k,l=0}^\infty a_k(f)a_l(g) T_k(t)T_l(t)\\
\hphantom{f(t)g(t)}{} = \sum_{k,l=0}^\infty a_k(f)a_l(g)\left\{\frac{1}{2}\cos [(k+l)\arccos t] +\frac{1}{2}\cos[|k-l|\arccos t]\right\},\!\!\qquad t \in [-1,1].
 \end{gather*}
Since the series above is absolute convergent at $t=1$, we can rearrange it in order to obtain
\begin{gather*}
f(t)g(t) = \frac{1}{2} \sum_{\mu=0}^\infty\bigg(\sum_{k+l=\mu}a_k(f)a_l(g)\bigg)T_{\mu}(t)\\
\hphantom{f(t)g(t) =}{} +\frac{1}{2} \sum_{\nu=0}^\infty\bigg(\sum_{|k-l|=\nu}a_k(f)a_l(g)\bigg)T_{\nu}(t), \qquad t\in [-1,1].
 \end{gather*}
That is,
\begin{gather*}f(t)g(t)=\sum_{m=0}^\infty a_m(fg) T_m(t),\qquad t \in [-1,1],\end{gather*}
in which
\begin{gather*}a_0(fg)=a_0(f)a_0(g)+\frac{1}{2}\sum_{\mu=1}^\infty a_\mu(f)a_\mu(g)\end{gather*}
and
\begin{gather*}a_m(fg)=\frac{1}{2}\sum_{\nu=0}^{m}a_\nu(f)a_{m-\nu}(g)+\frac{1}{2}\sum_{\mu=0}^\infty a_\mu(f)a_{\mu+m}(g)+a_{\mu+m}(f)a_\mu(g),\qquad m\geq 1.\end{gather*}
It is very easy to see that, for a fixed integer $m$, $a_{|m|}(fg)>0$ if, and only if, $m$ belongs to the set $\{\pm k\pm l\colon a_{k}(f)a_{l}(g)>0\}$. In view of Theorem \ref{SPD}$(i)$, we have proved the following result.

\begin{Theorem} \label{pros1} Let $f$ and $g$ be the isotropic parts of positive definite kernels on $S^1$ and consider their Fourier--Jacobi series representations according to~\eqref{PDM}. Then, $fg$ is the isotropic part of a strictly positive definite kernel on $S^1$ if, and only if, the set
\begin{gather*} \{\pm k \pm l\colon a_k(f)a_l(g)>0 \} \end{gather*}
intersects every full arithmetic progression in~$\mathbb{Z}$.
\end{Theorem}

Next, we will extend Theorem~\ref{pros1} to all the other compact two-point homogeneous spaces appearing in Wang's classification. We find convenient to prove and use the following technical result, a direct consequence of Lemma~\ref{gasp}.

\begin{Lemma}\label{technical} Let $\alpha$ and $\beta$ be real numbers with $\alpha \geq \beta >-1$ and $\alpha +\beta \geq -1$. Define $h\colon [-1,1] \to \mathbb{R}$ by the formula
\begin{gather*}h(t)= \sum_{k,l=0}^{\infty} b_{k,l} P_{k}^{(\alpha,\beta)}(t)P_{l}^{(\alpha,\beta)}(t), \qquad t \in [-1,1],\end{gather*}
where all the coefficients $b_{k,l}$ are nonnegative and $\sum\limits_{k,l=0}^{\infty} b_{k,l} P_{k}^{(\alpha,\beta)}(1)P_{l}^{(\alpha,\beta)}(1) < \infty$. The following assertions are equivalent:
\begin{enumerate}\itemsep=0pt
\item[$(i)$] the function $h$ is a polynomial;
\item[$(ii)$] the set $\{k+l\colon b_{k,l} >0\}$ is finite.
\end{enumerate}
\end{Lemma}
\begin{proof}The function $h$ is obviously well defined due to the inequality $\big|P_{k}^{(\alpha,\beta)}(t)\big|\leq P_{k}^{(\alpha,\beta)}(1)$, $t \in [-1,1]$. By Lemma~\ref{gasp}, we can put $h$ into the form
\begin{gather*}h(t)=\sum_{k,l=0}^{\infty} b_{k,l} \sum _{\mu=0}^{k+l}b_{k,l}^{\alpha, \beta}(\mu)P_{\mu}^{(\alpha,\beta)}(t)=
\sum_{\mu=0}^{\infty} \bigg[\sum_{k+l\geq \mu}b_{k,l}b_{k,l}^{\alpha, \beta}(\mu)\bigg] P_{\mu}^{(\alpha,\beta)}(t), \qquad t \in [-1,1],
\end{gather*}
where we are setting $b_{k,l}^{\alpha, \beta}(\mu)=0$ for $\mu\leq |k-l|-1$. If $h$ is a polynomial of degree $n$, then
\begin{gather*}0 = \sum_{k+l\geq \mu}b_{k,l}b_{k,l}^{\alpha, \beta}(\mu) \geq \sum_{k+l=\mu}b_{k,l}b_{k,l}^{\alpha, \beta}(\mu) \geq 0, \qquad \mu > n.\end{gather*}
In particular, $b_{k,l}=0$ when $k+l>n$. This shows that $(i)$ implies $(ii)$. The other implication is obvious.
\end{proof}

We are about ready for an extension of Theorem~\ref{pros1} to those compact two-point homogeneous spaces which are not spheres.

\begin{Theorem} \label{notsphere} Let $f$ and $g$ be isotropic parts of positive definite kernels on $\mathbb{M}^d$ and consider their Fourier--Jacobi series representations according to~\eqref{PDM}. Assume $\mathbb{M}^d$ is not a sphere. The function $fg$ is the isotropic part of a strictly positive definite kernel on~$\mathbb{M}^d$ if, and only if, the set $\{k+l\colon a_k(f)a_l(g)>0\}$ is infinite.
\end{Theorem}
\begin{proof} The starting point of the proof is the formula
\begin{gather*}f(t)g(t)=\sum_{k,l=0}^\infty a_k(f)a_l(g)P_{k}^{((d-2)/2,\beta)}(t)P_{l}^{((d-2)/2,\beta)}(t),\qquad t\in[-1,1],\end{gather*}
with $\beta$ depending upon the manifold~$\mathbb{M}^d$. An application of Theorem~\ref{SPD}$(iii)$ reveals that $fg$ is the isotropic part of a~positive definite kernel on $\mathbb{M}^d$ which is not strictly positive definite if, and only if, $fg$ is a polynomial. However, by Lemma~\ref{technical}, this is the case if, and only if, the set $\{k+l\colon a_k(f)a_l(g)>0\}$ is finite.
\end{proof}

The previous theorem allows the following reformulation.

\begin{Theorem} Let $f$ and $g$ be nonzero isotropic parts of positive definite kernels on $\mathbb{M}^d$. Assume~$\mathbb{M}^d$ is not a sphere. The following assertions are equivalent:
\begin{enumerate}\itemsep=0pt
\item[$(i)$] $fg$ is the isotropic part of a strictly positive definite kernel on $\mathbb{M}^d$;
\item[$(ii)$] either $f$ or $g$ is the isotropic part of a strictly positive definite kernel on~$\mathbb{M}^d$;
\item[$(iii)$] either $f$ or $g$ is not a polynomial.
\end{enumerate}
\end{Theorem}

We close the section handling the case in which $\mathbb{M}^d$ is a high-dimensional sphere.

\begin{Theorem} \label{sphere} $(d\geq 2)$. Let $f$ and $g$ be the isotropic parts of positive definite kernels on~$S^d$ and consider their Fourier--Jacobi series representations according to~\eqref{PDM}. The product~$fg$ is the isotropic part of a~strictly positive definite kernel on~$S^d$ if, and only if, the set
\begin{gather*}\{k+l\colon a_k(f)a_l(g)>0\}\end{gather*}
contains infinitely many even and infinitely many odd integers.
\end{Theorem}
\begin{proof} In the case in which $\alpha=\beta$, the Jacobi polynomials $P_k^{(\alpha,\beta)}$ are positive multiples of the Gegenbauer polynomials $C_k^{\alpha}$:
\begin{gather*}P_k^{(\alpha-1/2,\alpha-1/2)}(t)=\frac{(\alpha+1/2)_k}{(2\alpha)_k} C_k^\alpha(t),\qquad t \in [-1,1], \qquad k=0,1,\ldots.\end{gather*}
So, we can think of the expansions (\ref{PDM}) of $f$ and $g$ in terms of the Gegenbauer polyno\-mials~$C_k^{(d-1)/2}$. In particular, we may write
\begin{gather*}f(t)g(t)=\sum_{k,l=0}^\infty a_k(f)a_l(g)C_{k}^{(d-1)/2}(t)C_{l}^{(d-1)/2}(t),\qquad t\in[-1,1].\end{gather*}
By Theorem~\ref{SPD}$(ii)$, the function $fg$ is the isotropic part of a positive definite kernel on $S^d$ which is not strictly positive definite if, and only if, either $\{k\colon a_{2k}(fg)>0\}$ or $\{k\colon a_{2k+1}(fg)>0\}$ is finite. Since $C _{k}^{(d-1)/2}(-t)= (-1)^{k}C _{k}^{(d-1)/2}(t)$, $t \in [-1,1]$, we have that
\begin{gather*}f(t)g(t) -f(-t)g(-t)= 2\sum_{k=0}^{\infty}a_{2k+1}(fg)C_{2k+1}^{(d-1)/2}(t), \qquad t \in [-1,1],\end{gather*}
and
\begin{gather*}f(t)g(t) +f(-t)g(-t)= 2\sum_{k=0}^{\infty}a_{2k}(fg)C_{2k}^{(d-1)/2}(t),\qquad t \in [-1,1].\end{gather*}
Hence, the previous assertion corresponds to either one of the functions above being a polynomial. However, since we have the alternative representations
\begin{gather*}f(t)g(t) -f(-t)g(-t)= 2\sum_{k+l \in 2\mathbb{Z}_+ +1} a_k(f)a_l(g)C _{k}^{(d-1)/2}(t)C_{l}^{(d-1)/2}(t),\qquad t \in [-1,1],\end{gather*}
and
\begin{gather*}f(t)g(t) +f(-t)g(-t)= 2\sum_{k+l \in 2\mathbb{Z}_+} a_k(f)a_l(g)C _{k}^{(d-1)/2}(t)C_{l}^{(d-1)/2}(t),\qquad t \in [-1,1],\end{gather*}
we conclude from Lemma~\ref{technical}, that $fg$ is the isotropic part of a positive definite kernel on $S^d$ which is not strictly positive definite if, and only if, either $\{k+l \in 2\mathbb{Z}_++1\colon a_k(f)a_l(g)>0\}$ or $\{k+l\in 2\mathbb{Z}_+\colon a_k(f)a_l(g)>0 \}$ is finite. That is, the set
\begin{gather*}\{k+l\colon a_k(f)a_l(g)>0\}\end{gather*} contains either finitely many even or finitely many odd integers.
\end{proof}

\section{Extensions to intersection classes} \label{extend}

The setting here is still aligned with that adopted in the previous section. We will assume the isotropic part of the kernels have either one of the forms
\begin{gather}\label{PDM1}
f(t)=\sum_{k=0}^{\infty}a_k(f) t^k, \qquad t \in [-1,1],
\end{gather}
or
\begin{gather}\label{PDM2}
f(t)=\sum_{k=0}^{\infty}a_k(f) \left(\frac{1+t}{2}\right)^k, \qquad t \in [-1,1],
\end{gather}
where the coefficients $a_k(f)$ are all nonnegative and the series is convergent at $t=1$.

According to Schoenberg, a continuous function $f\colon [-1,1] \to \mathbb{R}$ admits the representation~(\ref{PDM1}) if, and only if, it is the isotropic part of a positive definite kernel on $S^d$, for $d=1,2,\ldots$. On the other hand, a theorem proved in \cite{guella0} shows that a continuous function $f$ has the representation (\ref{PDM2}) if, and only if, each one of the following three equivalent assertions hold:
\begin{enumerate}\itemsep=0pt
\item[--] $f$ is the isotropic part of a positive definite kernel on $P^d(\mathbb{R})$, for $d=2,3,\ldots$;
\item[--] $f$ is the isotropic part of a positive definite kernel on $P^d(\mathbb{C})$, for $d=4,8,\ldots$;
\item[--] $f$ is the isotropic part of a positive definite kernel on $P^d(\mathbb{H})$, for $d=8,12,\ldots$.
\end{enumerate}

For coherence, we will say that $f$ is the isotropic part of a positive definite kernel on $S^\infty$ in the first case and on $P^\infty$ in the second case. In the first case, the notation makes perfect sense if we interpret $S^\infty$ as the unit sphere in the usual real space $\ell_2$ endowed with its usual distance defined by
\begin{gather*}|xy|:=2\arccos x\cdot y,\qquad x,y \in S^\infty,\end{gather*}
where $\cdot$ denotes the standard inner product of $\ell_2$. Indeed, as explained in \cite{bingham, faraut,schoen}, if $f$ is continuous in $[-1,1]$, the kernel $K(x,y)=f(\cos (|xy|/2))$, $x,y \in S^\infty$, is positive definite if, and only if, $f$ admits the representation~(\ref{PDM1}).

Regarding strict positive definiteness for kernels generated by functions in the two classes above, the result detached below holds. The proof of the first assertion can be found in~\cite{mene} while the proof of the other one can be obtained similarly.

\begin{Theorem} \label{SPD2} Let $f$ be the isotropic part of a positive definite kernel $K$ on $S^\infty$ $($respective\-ly,~$P^\infty)$. The kernel $K$ is strictly positive definite if, and only if, the set $\{k\colon a_{k}(f)>0\}$ defined from~\eqref{PDM1} $($respectively, \eqref{PDM2}$)$ contains infinitely many even and infinitely many odd integers $($respectively, infinitely many integers$)$.
\end{Theorem}

The reader will easily verify that Lemma \ref{technical} still holds true when we replace $P_k^{(\alpha,\beta)}(t)$ with either~$t^k$ or $2^{-k}(1+t)^k$. Indeed, the proofs are a lot easier in these cases due to the simpler structure of the linearization formulas
\begin{gather*}t^k t^l=t^{k+l}, \qquad t \in [-1,1,],\qquad k,l\in \mathbb{Z}_+,\end{gather*}
and
\begin{gather*}\left(\frac{1+t}{2}\right)^k \left(\frac{1+t}{2}\right)^l=\left(\frac{1+t}{2}\right)^{k+l},\qquad t\in [-1,1],\qquad k,l \in \mathbb{Z}_+.\end{gather*}

Equally simple adaptations in the proofs of Theorems~\ref{notsphere} and~\ref{sphere} justify the following two extensions of the main results proved in Section~\ref{s1}.

\begin{Theorem} \label{sphere1} Let $f$ and $g$ be the isotropic parts of positive definite kernels on $S^\infty$ and consider their series representations according~\eqref{PDM1}. The product~$fg$ is the isotropic part of a strictly positive definite kernel on~$S^\infty$ if, and only if, the set
\begin{gather*}\{k+l\colon a_k(f)a_l(g)>0\}\end{gather*}
contains infinitely many even and infinitely many odd integers.
\end{Theorem}

\begin{Theorem} Let $f$ and $g$ be nonzero isotropic parts of positive definite kernels on~$P^\infty$. The following assertions are equivalent:
\begin{enumerate}\itemsep=0pt
\item[$(i)$] $fg$ is the isotropic part of a strictly positive definite kernel on $P^\infty$;
\item[$(ii)$] either $f$ or $g$ is the isotropic part of a strictly positive definite kernel on the space;
\item[$(iii)$] either $f$ or $g$ is not a polynomial.
\end{enumerate}
\end{Theorem}

\section[The case of a locally compact group cross $S^d$]{The case of a locally compact group cross $\boldsymbol{S^d}$} \label{produto}

In this section, we expand a little bit the setting considered in the previous sections by analyzing the very same question in a case that includes space-time positive definite kernels. The setting here is the one in \cite{berg,guella}, a brief description of which is as follows.

If $G$ is a locally compact group with operation $*$ and neutral element $e$, we intend to consider positive definite kernels $K\colon G \times S^d \to \mathbb{C}$ that have the form
\begin{gather*}((u,x),(v,y)) \in \big(G \times S^d\big)^2 \to K((u,x),(v,y))=f\big(u^{-1} *v, \cos(|xy|/2)\big) \end{gather*}
in which $f$ is a complex continuous function with domain $G \times [-1,1]$. In this setting, we need to use the definition of positive definiteness in its full strength, that is, the left hand side of~(\ref{ineq}) needs to be of the form $\sum\limits_{i,j=1}^n c_i \overline{c_j} K(y_i,y_j)$, with complex $c_i$ and the $y_i$ in $G \times S^d$. A similar remark holds for strict positive definiteness. In analogy with the previous sections, we will call the function $f\colon G \times [-1,1]\to \mathbb{C}$, the isotropic part of the kernel. According to~\cite{berg}, the positive definiteness of $K$ corresponds to the following series representation for $f$
\begin{gather*}f(u,t)=\sum_{k=0}^{\infty}a_k^d(f;u) C_k^{(d-1)/2}(t), \qquad (u,t) \in G \times [-1,1],\end{gather*}
in which $\{a_k^d(f;\cdot)\}$ is a sequence of continuous functions on $G$ defining positive definite kernels $(u,v)\in G^2 \to a_k^d\big(f;u^{-1}*v\big)$ and $\sum\limits_{k} a_k^d(f;e)C_k^{(d-1)/2}(1)<\infty$. The functions $a_k^d(f;\cdot)$ appearing above have a closed integral form given by
\begin{gather*}a_k^d(f;u)=c(k,d)\int_{-1}^1 f(u,s)C_k^{(d-1)/2}(s) \big(1-s^2\big)^{(d-2)/2}{\rm d}s, \qquad u \in G,\end{gather*}
in which~$c(k,d)$ is an appropriate normalization constant depending upon~$k$ and~$d$.

The following result concerning the strict positive definiteness of a kernel fitting the description presented in the previous paragraph is a~consequence of a quite general result proved in~\cite{guella}. It boils down to strict positive definiteness on~$S^d$ of a large class of positive definite functions indexed over~$G$.

\begin{Theorem}[$d\geq 2$] \label{keyy} Let $f$ be the isotropic part of a positive definite kernel on $G \times S^d$. The following assertions are equivalent:
\begin{enumerate}\itemsep=0pt
\item[$(i)$] the kernel $((u,x),(v,y)) \in \big(G \times S^d\big)^2 \to f\big(u^{-1}*v, \cos(|xy|/2)\big)$ is strictly positive definite;
\item[$(ii)$] if $p$ is a positive integer at most the cardinality of~$G$, $u_1, u_2, \ldots, u_p$ are distinct points in~$G$ and $c$ is a~nonzero vector in $\mathbb{C}^p$, then the function
\begin{gather*} t \in [-1,1] \to \sum_{\mu,\nu=1}^{p}c_{\mu}\overline{c_{\nu}} f\big(u_{\mu}^{-1}*v_{\nu}, t\big), \end{gather*}
is the isotropic part of a strictly positive definite kernel on~$S^{d}$.
\end{enumerate}
\end{Theorem}

The main result in this section is as follows.

\begin{Theorem}[$d\geq 2$] \label{keyy1} Let $f$ and $g$ be isotropic parts of positive definite kernels on $G \times S^d$. The following assertions are equivalent:
\begin{enumerate}\itemsep=0pt
\item[$(i)$] the product $fg$ is the isotropic part of a strictly positive definite kernel on $G\times S^d$;
\item[$(ii)$] the set
\begin{gather*}
\big\{k+l\colon c^t \big[a_k^d\big(f;u_\mu^{-1}* u_\nu\big)a_l^d\big(g;u_\mu^{-1} * u_\nu \big)\big]_{\mu , \nu =1}^{p} \overline{c}>0\big\}
\end{gather*}
contains infinitely many even and infinitely many odd integers, whenever $p\geq 1$, $u_1, u_2,$ $\ldots,u_p$ are distinct points in~$G$ and $c$ is a nonzero vector in~$\mathbb{C}^p$.
\end{enumerate}
\end{Theorem}
\begin{proof} By Theorem~\ref{keyy}, the function $fg$ is the isotropic part of a strictly positive definite kernel on $G\times S^{d}$ if, and only if, for every $p\geq 1$, distinct points $u_1, u_2, \ldots, u_p$ in $G$ and a nonzero vec\-tor~$c$ in~$\mathbb{C}^p$, the function $h\colon [-1,1] \to \mathbb{C} $ defined by
\begin{gather*} h(t)= \sum_{\mu,\nu=1}^{p}c_{\mu}\overline{c_{\nu}} f\big(u_{\mu}^{-1}*v_{\nu}, t\big)g\big(u_{\mu}^{-1}*v_{\nu}, t\big)\end{gather*}
is the isotropic part of a strictly positive definite kernel on~$S^{d}$. Introducing series representations, we can re-write the expression above as
\begin{gather*}h(t)=\!\sum_{k,l=0}^{\infty}\! \left[ \sum_{\mu,\nu=1}^{p} c_{\mu}\overline{c_{\nu}} a_{k}^d\big(f,u_{\mu}^{-1}*v_{\nu}\big) a_{l}^d\big(g,u_{\mu}^{-1}*v_{\nu} \big) \right]\! C_{k}^{(d-1)/2}(t)C_{l}^{(d-1)/2}(t),\!\!\!\qquad t \in [-1,1].\end{gather*}
In other words,
\begin{gather*}h(t)=\!\sum_{k,l=0}^{\infty}\! \big\{c^t \big[a_k^d\big(f;u_\mu^{-1}* u_\nu\big)a_l^d\big(g;u_\mu^{-1} * u_\nu \big)\big]_{\mu , \nu =1}^{p} \overline{c}\big\} C_{k}^{(d-1)/2}(t)C_{l}^{(d-1)/2}(t), \!\!\!\!\qquad t \in [-1,1], \end{gather*}
in which $c$ is the vector with entries $c_1, c_2, \ldots, c_p$. We now can repeat the arguments used in Section~\ref{extend} in order to see that~$h$ is the isotropic part of a strictly positive definite kernel on~$S^{d}$ if, and only if,
\begin{gather*}\big\{k+l\colon c^t \big[a_k^d\big(f;u_\mu^{-1}* u_\nu\big)a_l^d\big(g;u_\mu^{-1} * u_\nu \big)\big]_{\mu , \nu =1}^{p} \overline{c}>0\big\}\end{gather*}
contains infinitely many even and infinitely many odd integers.
\end{proof}

The previous theorem can be put in a more general form, following the setting adopted in~\cite{guella}. Details will be not included here.

The strict positive definiteness of product covariance functions on $\mathbb{R}^d$ alone was considered in \cite{posa,iaco,podeiaco}, where the reader can also find explanations regarding the practicability and the computational advantages and simplifications implied by the use of such separable covariance functions in the geostatistical literature. However, a self-contained characterization for the strict positive definiteness of a product of positive definite kernels on $\mathbb{R}^d$ is still elusive.

We close the section presenting an explicit construction in the case in which $G$ is the usual group $(\mathbb{R},+)$.

\begin{Example}
Let $f$ and $g$ be the isotropic parts of strictly positive definite kernels on~$S^{\infty}$. If~$\lambda$ and~$\theta$ are real numbers, then the functions $F$ and $G$ given by
\begin{gather*}
F(u,t)=f(t \cos \lambda u )=\sum_{k=0}^\infty \big[a_k(f)\cos^k \lambda u\big] t^k, \qquad (t,u) \in [-1,1] \times \mathbb{R},
\end{gather*}
and
\begin{gather*}
G(u,t)=g(t \cos \theta u)=\sum_{k=0}^\infty \big[a_l(g)\cos^l \theta u\big] t^l , \qquad (t, u) \in [-1,1] \times \mathbb{R},
\end{gather*}
are isotropic parts of positive definite kernels on $\mathbb{R} \times S^{\infty}$ and these kernels are not strictly positive definite. However, if~$\lambda$ and~$\theta$ are nonzero and $\lambda/\theta$ is not a~rational number, then~$FG$ is the isotropic part of a strictly positive definite kernel on~$\mathbb{R} \times S^{\infty}$. Indeed, if $p\geq 1$, $u_1, u_2$, $\ldots, u_p$ are distinct points in~$\mathbb{R}$ and~$c$ is a nonzero vector in~$\mathbb{C}^p$, then
\begin{gather*}
\lim_{\min (k,l) \to \infty} \cos(\lambda (x_{\mu} - x_{\nu}))^{k} \cos(\theta (x_{\mu} - x_{\nu}))^{l} = \delta_{\mu\nu}.
\end{gather*}
In particular, the set
\begin{gather*}
\big\{k+l\colon c^t \big[a_k(f)\cos(\lambda (x_{\mu} - x_{\nu}))^{k} a_l(g) \cos(\theta (x_{\mu} - x_{\nu}))^{l} \big]_{\mu , \nu =1}^{p} \overline{c}>0\big\}
\end{gather*}
contains infinitely many even and infinitely many odd integers, as required by Theorem \ref{keyy1}$(ii)$.
\end{Example}

\section{The complex spherical case} \label{complex}

In this section, we revisit Section~\ref{s1} and the complex counterpart of the spherical case.

We write $\Omega_{2q}$ to denote the unit sphere in $\mathbb{C}^q$ and $\cdot$ to denote the usual inner product of $\mathbb{C}$. The positive definite kernels to be considered here are of the form
\begin{gather*}K(x,y)=f(x \cdot y), \qquad x,y \in \Omega_{2q},\end{gather*}
in which $f$ is a complex continuous function on $\mathbb{D}:=\{z\in \mathbb{C}\colon |z|\leq 1\}$ if $q\geq 2$ and on $\Omega_2$ if $q=1$. Once again, we observe that the scalars in the definition of positive definiteness according to~(\ref{ineq}) needs to be complex ones. A positive definite kernel on~$\Omega_{2q}$ is invariant with respect to unitary transformations of $\mathbb{C}^q$, the reason why we will call~$f$ the isotropic part of~$K$.

According to~\cite{menepe}, in the case $q\geq 2$, a kernel $K$ as above is positive definite if, and only if, its isotropic part $f$ is representable in the form
\begin{gather}\label{repC}
f(z)= \sum_{m,n=0}^\infty a_{m,n}(f) R_{m,n}^{q-2}(z), \qquad z \in \mathbb{D},
\end{gather}
in which all the coefficients $a_{m,n}$ are nonnegative, $R_{m,n}^{q-2}$ is the disk (or Zernike) polynomial of bi-degree $(m,n)$ associated with the integer $q-2$ and the series is convergent at $z=1$. The disk polynomials are discussed in~\cite{wunsche}. In the case $q=1$, the representation for $f$ becomes
\begin{gather*}
f(z)=\sum_{m\in \mathbb{Z}} a_m(f) z^m, \qquad z \in \Omega_2,
\end{gather*}
in which all the coefficients $a_m(f)$ are nonnegative and the series is convergent at $z=1$. The strict positive definiteness of the kernel in each case is equivalent to $\{m-n\colon a_{m,n}(f)>0\}$ (respectively, $\{m\colon a_m(f)>0\}$) intersecting every full arithmetic progression of $\mathbb{Z}$. Details on that can be found in~\cite{guella2,menega}.

If $f$ and $g$ are the isotropic parts of two positive definite kernels on $\Omega_2$, a procedure very close to that used at the beginning in Section~\ref{extend} leads to the following criterion: $fg$ is the isotropic part of a strictly positive definite kernel on $\Omega_{2}$ if, and only if, $\{m+n\colon a_m(f)a_n(g)>0\}$ intersects every full arithmetic progression of $\mathbb{Z}$. The details will be not included.

In order to handle the case $q\geq2$, it is relevant to recall a linearization formula for disk polynomials proved by Koornwinder~\cite{koorn}, a generalization of the one described in Lemma~\ref{gasp}: for nonnegative integers $m_1$, $m_2$, $n_1$, $n_2$, it reads
\begin{gather} \label{lindisc}
R_{m_{1},n_{1}}^{q-2}(z)R_{m_{2},n_{2}}^{q-2}(z)= \sum_{m, n} a_{m_{1},n_{1};m_{2},n_{2}}^{q;m,n}R_{m,n}^{q-2}(z),\qquad z\in \mathbb{D},
\end{gather}
in which all the coefficients $a_{m_{1},n_{1};m_{2},n_{2}}^{q;m,n}$ are nonnegative. The sum takes into account just the pairs $(m,n)$ satisfying
\begin{gather*}m_1+m_2 +n= n_1 + n_2 +m, \qquad |m_1+n_1-m_2-n_2|\leq m+n\leq m_1+n_1+m_2+n_2.\end{gather*}
Its structure is not as good as the other ones we have used so far in this paper, once the disk polynomials are double-indexed functions. However, a counterpart of Lemma~\ref{technical} can be enunciated and proved as follows.

\begin{Lemma}\label{koornbas}
Let $f$ be a function as in \eqref{repC}. For each $k \in \mathbb{Z}$, define
\begin{gather*}f_k(z)=\sum_{m-n=k}a_{m,n}(f) R_{m,n}^{q-2}(z), \qquad z \in \mathbb{D}.
\end{gather*}
The following assertions are equivalent:
\begin{enumerate}\itemsep=0pt
\item[$(i)$] the function $f$ is the isotropic part of a strictly positive definite kernel on $\Omega_{2q}$;
\item[$(ii)$] the set $\{k\colon f_k\not \equiv 0\}$ intersects every full arithmetic progression in $\mathbb{Z}$.
\end{enumerate}
\end{Lemma}
\begin{proof} Since the series expansion for $f$ is convergent at $z=1$, and $\big|R_{m,n}^{q-2}(z)\big| \leq R_{m,n}^{q-2}(1)$, we can write the equality
\begin{gather*}f(z)=\sum_{k\in \mathbb{Z}} f_k(z), \qquad z\in \mathbb{D}.\end{gather*}
Obviously, each $f_k$ is continuous in $\mathbb{D}$ and, in addition, it is the isotropic part of a positive definite kernel on~$\Omega_{2q}$. Since
\begin{gather*}\{m-n\colon a_{m,n}(f)>0\}=\{k\colon f_k\not \equiv 0\},\end{gather*}
the assertion in the statement of the lemma follows.
\end{proof}

We can now state and prove the main result in this section.

\begin{Theorem}[$q\geq 2$] Let $f$ and $g$ be the isotropic parts of two positive definite kernels on~$\Omega_{2q}$ and consider their series representation according to~\eqref{repC}. Then, $fg$ is the isotropic part of a~strictly positive definite kernel on~$\Omega_{2q}$ if, and only if, the set
\begin{gather*} \{(m-n) + (m'-n')\colon a_{m,n}(f)a_{m',n'}(g)>0 \} \end{gather*}
intersects every full arithmetic progression in $\mathbb{Z}$.
\end{Theorem}
\begin{proof}The first step in the proof is to write
\begin{gather*}f(z)g(z) =\sum_{k,l \in \mathbb{Z}} f_{k}(z)g_{l}(z)= \sum _{\mu \in \mathbb{Z} }\sum_{k+l=\mu} f_{k}(z)g_{l}(z),\qquad z \in \mathbb{D}.\end{gather*}
The second equality above is supported by the inequality
\begin{gather*}
\sum_{k,l \in \mathbb{Z}} |f_{k}(z)g_{l}(z)|\leq \sum_{k,l \in \mathbb{Z}} f_{k}(1)g_{l}(1)=f(1)g(1).
\end{gather*}
Since the indices in (\ref{lindisc}) satisfy
\begin{gather*}m-n= (m_{1}-n_{1} ) + (m_{2}-n_{2}),\end{gather*}
for each pair $((m_1,n_1),(m_2,n_2))$, we can write
\begin{gather*}
f_{k}(z)g_{l}(z)=\sum_{m-n=k+l}b^{k,l}_{m,n}R_{m,n}^{q-2}(z),\qquad z\in \mathbb{D},\end{gather*}
where all the coefficients $b^{k,l}_{m,n}$ are nonnegative. Consequently,
\begin{gather*}\sum_{k+l=\mu}f_k(z)g_l(z)=\sum_{m-n=\mu}b_{m,n}R_{m,n}^{q-2}(z),\qquad z\in \mathbb{D},\end{gather*}
and, in particular,
\begin{gather*}(fg)_\mu(z)=\sum_{k+l=\mu}f_k(z)g_l(z),\qquad z \in \mathbb{D}.\end{gather*}
Thus, by Lemma~\ref{koornbas}, it remains to show that
\begin{gather*}\{\mu\colon (fg)_\mu \not \equiv 0 \}= \{ m-n+m'-n'\colon a_{m,n}(f)a_{m',n'}(g)>0\}.\end{gather*}
If $(fg)_\mu$ is not the zero function for some $\mu$, then there exists a pair $(k,l)$ with $k+l=\mu$ so that neither $f_k$ nor $g_l$ is the zero function. Hence, there are pairs $(m,n)$ with $m-n=k$ and $(m',n')$ with $m'-n'=l$ so that $a_{m,n}(f)>0$ and $a_{m',n'}(g)>0$. In other words, $\mu \in \{m-n+m'-n'\colon a_{m,n}(f)a_{m',n'}(g)>0\}$. Conversely, if $\mu \in \{m-n+m'-n'\colon a_{m,n}(f)a_{m',n'}(g)>0\}$, then $\mu=m-n+m'-n'$ for pairs $(m,n)$ and $(m',n')$ for which $a_{m,n}(f)a_{m',n'}(g)>0$. Defining, $k=m-n$ and $l=m'-n'$, it follows that
both $f_k$ and $g_l$ are not identically 0. In particular, $f_k(1)>0$ and $g_l(1)>0$, whence $(fg)_\mu(1)=(fg)_{m-n+m'-n'}(1)>0$. Thus, $(fg)_\mu$ is not identically zero.
\end{proof}

\subsection*{Acknowledgements}

The first and third authors acknowledge partial financial support from FAPESP, grants $\#$ 2014/14380-2 and 2016/09906-0. Likewise, the second one acknowledges partial financial support from PROEX-CAPES. The authors would like to express their gratitude to the anonymous referees for their valuable suggestions on the original manuscript.

\pdfbookmark[1]{References}{ref}
\LastPageEnding


\begin{thebibliography}{99}
\footnotesize\itemsep=0pt

\bibitem{askey}
Askey R., Orthogonal polynomials and special functions, Society for Industrial
 and Applied Mathematics, Philadelphia, Pa., 1975.

\bibitem{barbosa}
Barbosa V.S., Menegatto V.A., Strictly positive definite kernels on compact
 two-point homogeneous spaces, \href{https://doi.org/10.7153/mia-19-54}{\textit{Math. Inequal. Appl.}} \textbf{19}
 (2016), 743--756, \href{https://arxiv.org/abs/1505.00591}{arXiv:1505.00591}.

\bibitem{barbosa0}
Barbosa V.S., Menegatto V.A., Strict positive definiteness on products of
 compact two-point homogeneous spaces, \href{https://doi.org/10.1080/10652469.2016.1249867}{\textit{Integral Transforms Spec.
 Funct.}} \textbf{28} (2017), 56--73, \href{https://arxiv.org/abs/1605.07071}{arXiv:1605.07071}.

\bibitem{berg}
Berg C., Porcu E., From {S}choenberg coefficients to {S}choenberg functions,
 \href{https://doi.org/10.1007/s00365-016-9323-9}{\textit{Constr. Approx.}} \textbf{45} (2017), 217--241, \href{https://arxiv.org/abs/1505.05682}{arXiv:1505.05682}.

\bibitem{bingham}
Bingham N.H., Positive definite functions on spheres, \href{https://doi.org/10.1017/S0305004100047551}{\textit{Proc. Cambridge
 Philos. Soc.}} \textbf{73} (1973), 145--156.

\bibitem{chen}
Chen D., Menegatto V.A., Sun X., A necessary and sufficient condition for
 strictly positive definite functions on spheres, \href{https://doi.org/10.1090/S0002-9939-03-06730-3}{\textit{Proc. Amer. Math.
 Soc.}} \textbf{131} (2003), 2733--2740.

\bibitem{posa}
De~Iaco S., Myers D.E., Posa D., On strict positive definiteness of product and
 product-sum covariance models, \href{https://doi.org/10.1016/j.jspi.2010.09.014}{\textit{J.~Statist. Plann. Inference}}
 \textbf{141} (2011), 1132--1140.

\bibitem{iaco}
De~Iaco S., Myers D.E., Posa D., Strict positive definiteness of a product of
 covariance functions, \href{https://doi.org/10.1080/03610926.2010.513790}{\textit{Comm. Statist. Theory Methods}} \textbf{40}
 (2011), 4400--4408.

\bibitem{podeiaco}
De~Iaco S., Posa D., Strict positive definiteness in geostatistics,
 \href{https://doi.org/10.1007/s00477-017-1432-x}{\textit{Stoch. Environ. Res. Risk Assess.}} \textbf{32} (2018),
 577--590.

\bibitem{faraut}
Faraut J., Fonction brownienne sur une vari\'{e}t\'{e} riemannienne, in
 S\'{e}minaire de {P}robabilit\'{e}s, {VII} ({U}niv. {S}trasbourg, ann\'{e}e
 universitaire 1971--1972), \href{https://doi.org/10.1007/BFb0071397}{\textit{Lecture Notes in Math.}}, Vol.~321,
 Springer, Berlin, 1973, 61--76.

\bibitem{gangolli}
Gangolli R., Positive definite kernels on homogeneous spaces and certain
 stochastic processes related to {L}\'{e}vy's {B}rownian motion of several
 parameters, \textit{Ann. Inst. H.~Poincar\'{e} Sect.~B} \textbf{3} (1967),
 121--226.

\bibitem{gasper}
Gasper G., Linearization of the product of {J}acobi polynomials.~{I},
 \href{https://doi.org/10.4153/CJM-1970-020-2}{\textit{Canad.~J. Math.}} \textbf{22} (1970), 171--175.

\bibitem{gasper1}
Gasper G., Linearization of the product of {J}acobi polynomials.~{II},
 \href{https://doi.org/10.4153/CJM-1970-065-4}{\textit{Canad.~J. Math.}} \textbf{22} (1970), 582--593.

\bibitem{guella1}
Guella J.C., Menegatto V.A., Strictly positive definite kernels on a product of
 spheres, \href{https://doi.org/10.1016/j.jmaa.2015.10.026}{\textit{J.~Math. Anal. Appl.}} \textbf{435} (2016), 286--301,
 \href{https://arxiv.org/abs/1505.03695}{arXiv:1505.03695}.

\bibitem{guella0}
Guella J.C., Menegatto V.A., A limit formula for semigroups defined by
 {F}ourier--{J}acobi series, \href{https://doi.org/10.1090/proc/13889}{\textit{Proc. Amer. Math. Soc.}} \textbf{146}
 (2018), 2027--2038.

\bibitem{guella2}
Guella J.C., Menegatto V.A., Unitarily invariant strictly positive definite
 kernels on spheres, \href{https://doi.org/10.1007/s11117-017-0502-0}{\textit{Positivity}} \textbf{22} (2018), 91--103.

\bibitem{guella}
Guella J.C., Menegatto V.A., Schoenberg's theorem for positive definite
 functions on products: a unifying framework, \href{https://doi.org/10.1007/s00041-018-9631-5}{\textit{J.~Fourier Anal. Appl.}},
 {t}o appear.

\bibitem{helgason}
Helgason S., Groups and geometric analysis. Integral geometry, invariant
 differential operators, and spherical functions, \href{https://doi.org/10.1090/surv/083}{\textit{Mathematical Surveys
 and Monographs}}, Vol.~83, Amer. Math. Soc., Providence, RI, 2000.

\bibitem{horn}
Horn R.A., Johnson C.R., Matrix analysis, Cambridge University Press,
 Cambridge, 1990.

\bibitem{hylle}
Hylleraas E.A., Linearization of products of {J}acobi polynomials,
 \href{https://doi.org/10.7146/math.scand.a-10527}{\textit{Math. Scand.}} \textbf{10} (1962), 189--200.

\bibitem{koorn}
Koornwinder T., Positivity proofs for linearization and connection coefficients
 of orthogonal polynomials satisfying an addition formula, \href{https://doi.org/10.1112/jlms/s2-18.1.101}{\textit{J.~London
 Math. Soc.}} \textbf{18} (1978), 101--114.

\bibitem{mene}
Menegatto V.A., Strictly positive definite kernels on the {H}ilbert sphere,
 \href{https://doi.org/10.1080/00036819408840292}{\textit{Appl. Anal.}} \textbf{55} (1994), 91--101.

\bibitem{menega}
Menegatto V.A., Oliveira C.P., Peron A.P., Strictly positive definite kernels
 on subsets of the complex plane, \href{https://doi.org/10.1016/j.camwa.2006.04.006}{\textit{Comput. Math. Appl.}} \textbf{51}
 (2006), 1233--1250.

\bibitem{menepe}
Menegatto V.A., Peron A.P., Positive definite kernels on complex spheres,
 \href{https://doi.org/10.1006/jmaa.2000.7264}{\textit{J.~Math. Anal. Appl.}} \textbf{254} (2001), 219--232.

\bibitem{schoen}
Schoenberg I.J., Positive definite functions on spheres, \href{https://doi.org/10.1215/S0012-7094-42-00908-6}{\textit{Duke Math.~J.}}
 \textbf{9} (1942), 96--108.

\bibitem{szego}
Szeg\H{o} G., Orthogonal polynomials, \textit{American Mathematical Society,
 Colloquium Publications}, Vol.~23, 4th~ed., Amer. Math. Soc., Providence,
 R.I., 1975.

\bibitem{wang}
Wang H.-C., Two-point homogeneous spaces, \href{https://doi.org/10.2307/1969427}{\textit{Ann. of Math.}} \textbf{55}
 (1952), 177--191.

\bibitem{wolf}
Wolf J.A., Spaces of constant curvature, 6th ed., AMS Chelsea Publishing,
 Providence, RI, 2011.

\bibitem{wunsche}
W\"{u}nsche A., Generalized {Z}ernike or disc polynomials, \href{https://doi.org/10.1016/j.cam.2004.04.004}{\textit{J.~Comput.
 Appl. Math.}} \textbf{174} (2005), 135--163.

\end{thebibliography}
\end{document}